\documentclass[12pt]{article}
\usepackage{amsmath,amsfonts,amsthm,amssymb,ytableau,multirow}

\usepackage[margin=1in]{geometry}
\usepackage{url}
\usepackage{graphicx}
\usepackage{float}
\usepackage{verbatim}

\newtheorem{thm}{Theorem}
\newtheorem{lem}[thm]{Lemma}

\newtheorem{ex}[thm]{Example}
\newtheorem{dfn}[thm]{Definition}
\newtheorem{remark}[thm]{Remark}
\numberwithin{equation}{section}
\newdimen\Squaresize \Squaresize=14pt
\newdimen\Thickness \Thickness=0.4pt

\def\Square#1{\hbox{\vrule width \Thickness
   \vbox to \Squaresize{\hrule height \Thickness\vss
      \hbox to \Squaresize{\hss#1\hss}
   \vss\hrule height\Thickness}
\unskip\vrule width \Thickness} \kern-\Thickness}

\def\Vsquare#1{\vbox{\Square{$#1$}}\kern-\Thickness}

\title{Correlations of minimal forbidden factors of the Fibonacci word}
\author{Narad Rampersad and Max Wiebe\footnote{
Department of Math/Stats,
University of Winnipeg,
515 Portage Ave.,
Winnipeg, MB, R3B 2E9
Canada; {\tt narad.rampersad@gmail.com}.}}

\begin{document}
\maketitle
\begin{abstract}
If $u$ and $v$ are two words,
the correlation of $u$ over $v$ is a binary word that encodes all possible overlaps
between $u$ and $v$.  This concept was introduced by Guibas and Odlyzko as a key element
of their method for enumerating the number of words of length $n$ over a given alphabet
that avoid a given set of forbidden factors.  In this paper we characterize the pairwise
correlations between the minimal forbidden factors of the infinite Fibonacci word.
\end{abstract}

\section{Introduction}
The infinite Fibonacci word
$${\bf f} = 010010100100101001010010\cdots$$
is one of the most well-studied infinite words.  One property of an infinite word ${\bf x}$ that is often
analyzed is the set of its forbidden factors, i.e., the set of factors that \emph{do not appear}
in the word.  In this analysis, one can restrict oneself to the minimal forbidden factors, i.e.,
the forbidden factors $u$ with the property that no proper factor of $u$ is forbidden in ${\bf x}$.
The paper of Mignosi, Restivo, and Sciortino~\cite{MRS02} is a good introduction to this topic.
In particular, it includes a characterization of the set $M$ of minimal forbidden factors of the Fibonacci word:
$$M = \{11, 000, 10101, 00100100, 1010010100101, \ldots\}.$$

When studying the avoidance of words, one often wishes to enumerate the number of words of length $n$
over a given alphabet that avoid a finite set of words $S$.  Guibas and Odlyzko~\cite{GO81} gave a
generating function method based on \emph{correlations}:  if $u$ and $v$ are two words, the correlation
of $u$ over $v$ is a binary word that encodes all possible overlaps between $u$ and $v$.  The main result of
this paper is a characterization of the pairwise correlations between the minimal forbidden factors of the
Fibonacci word.

\section{Background information}\label{sec1}
The sequence of Fibonacci numbers is defined by $F_1 = F_2 = 1$, and $F_n = F_{n-1} + F_{n-2}$, for every $n>2$. The first few values of the sequence $F_n$ are given in Table~\ref{tab1}.
\begin{table}[H]
    \centering
    \begin{tabular}{c|c|c|c|c|c|c|c|c|c|c|c|c}
        $n$ & 1 & 2 & 3 & 4 & 5 & 6 & 7 & 8 & 9 & 10 & 11 & 12 \\
        \hline
        $F_n$ & 1 & 1 & 2 & 3 & 5 & 8 & 13 & 21 & 34 & 55 & 89 & 144 
    \end{tabular}
    \caption{The first few values of the sequence of Fibonacci numbers.}
    \label{tab1}
\end{table}
Let $A$ be a finite alphabet and $A^*$ be the set of all finite words over $A$, including the empty word $\epsilon$, and let $A^n$ be the set of all words over $A$ of length $n$. For notation purposes, we will use [$n$] = $\{0,1,\cdots,n-1\}$. Let $L \subseteq A^*$ be a $\textit{factorial language}$, i.e., a language satisfying:
\begin{align*}
    \forall u,v \in A^*,\; uv\in L \implies u, v \in L
\end{align*}
A word $v\in A^*$ is \emph{forbidden} for the factorial language $L$ if $v\not\in L$, which is equivalent to saying that $v$ occurs in no word of $L$. In addition, $v$ is \emph{minimal} if it has no proper factor that is forbidden.
Let $M(L)$ denote the set of \emph{minimal forbidden words} for $L$. 
\begin{remark}\label{Rem1}
    A word of length n, $v = v[0,1,...,n-1]$ belongs to $M(L)$ if and only if two conditions hold:
    \begin{itemize}
        \item $v\not\in L$, (i.e. $v$ is forbidden),
        \item $v[0,1,\cdots,n-2]$ and $v[1,2,\cdots,n-1] \in L$ (i.e. the prefix and suffix of $v$ of length $n-1$ belongs to $L$)
    \end{itemize}
\end{remark}
The  $\textit{Fibonacci infinite word}$ $\textbf{f}$ over the binary alphabet $\{0,1\}$ is the fixed point $\phi^\omega(0)$ of the Fibonacci morphism $\phi$, defined by
\begin{equation}\label{phi}
    \phi(0) = 01, \;\; \phi(1) = 0
\end{equation}
Let us also define the $n^{th} \textit{Fibonacci word}$, $f_n$, defined by $f_1 = 1$, $f_2 = 0$,  \\
\begin{equation}\label{fn}
    f_{n+2} = f_{n+1}f_n, \;\; n\geq1
\end{equation}
or equivalently, 
\begin{equation}\label{fn:phi}
    f_{n+2} = \phi^n(0), \;\;\text{and}\;\; f_{n+1} = \phi(f_n), \;\; n\geq1
\end{equation}
The Fibonacci words also have the property that the length of $f_n$, denoted $|f_n|$, is $|f_n| = F_n$ for each $n$. We record the first few Fibonacci words in Table \ref{tab2}.
\begin{table}[H]
    \centering
    \begin{tabular}{l}
         $f_1$ = 1 \\
         $f_2$ = 0 \\
         $f_3$ = 01 \\
         $f_4$ = 010 \\
         $f_5$ = 01001 \\
         $f_6$ = 01001010 \\
         $f_7$ = 0100101001001 \\
         $f_8$ = 010010100100101001010 \\
         $f_9$ = 0100101001001010010100100101001001
    \end{tabular}
    \caption{The first few Fibonacci words.}
    \label{tab2}
\end{table}



The $\textit{reverse}$ of a word $v = v[0,1,\cdots,n-1]$ is the word $\Tilde{v} = v[n-1,n-2,\cdots,0]$. If a word $v$ equals its reverse, $v$ is a $\textit{palindrome}$. If $v = xb = by$, we say $b$ is a $\textit{border}$ of $v$, or that $b$ $\textit{borders}$ $v$. For example, the word $v = 01001010010$ has borders $0$, $010$, $010010$, and $v$. Note that borders may overlap themselves and that every word borders itself.

Let us consider the sequence $p_n$ of the $\textit{palindromic prefixes}$ of $\textbf{f}$. The first few values of the sequence $p_n$ are recorded in Table~\ref{tab3}.
\begin{table}[H]
    \centering
    \begin{tabular}{l}
         $p_3$ = $\epsilon$ \\
         $p_4$ = 0 \\
         $p_5$ = 010 \\
         $p_6$ = 010010 \\
         $p_7$ = 01001010010 \\
         $p_8$ = 0100101001001010010 \\
         $p_9$ = 01001010010010100101001001010010
    \end{tabular}
    \caption{The first few palindromic prefixes $p_n$ of $\textbf{f}$.}
    \label{tab3}
\end{table}
For every $n\geq3$, $p_n$ is obtained from $f_n$ by removing the last two symbols. More specifically, for every $n\geq1$, 
\begin{equation}\label{fn:pn}
    f_{2n+1} = p_{2n+1}01, \quad f_{2n+2} = p_{2n+2}10
\end{equation}

\begin{lem}\label{pn:phi}
For $n \geq 3$, we have $\phi(p_n)0 = p_{n+1}$.
\end{lem}

\begin{proof}
If $n$ is odd, write $n=2m+1$.  Using \eqref{fn:pn}, we have
\begin{align*}
\phi(p_{2m+1})0&=\phi(f_{2m+1}(01)^{-1})0\\
&=\phi(f_{2m+1})(010)^{-1}0\\
&=f_{2m+2}(10)^{-1}\\
&=p_{2m+2}10(10)^{-1}\\
&=p_{2m+2}.
\end{align*}

If $n$ is even, write $n=2m+2$.  Using \eqref{fn:pn}, we have
\begin{align*}
\phi(p_{2m+2})0&=\phi(f_{2m+2}(10)^{-1})0\\
&=\phi(f_{2m+2})(001)^{-1}0\\
&=f_{2m+3}(01)^{-1}\\
&=p_{2m+3}01(01)^{-1}\\
&=p_{2m+3}.
\end{align*}
\end{proof}

The fundamental property of the palindromic prefixes of $\textbf{f}$ is the following:
\begin{lem}\label{lem1} For $n\geq1$,
    \begin{align*}
        p_{2n+1} = p_{2n-1}01p_{2n} = p_{2n}10p_{2n-1}, \;\;\;\;\;\; p_{2n+2} = p_{2n}10p_{2n+1} = p_{2n+1}01p_{2n}
    \end{align*}
\end{lem}
\begin{proof}
    Follows immediately from \eqref{fn:pn} and the fact that $f_n = f_{n-1}f_{n-2}$ for $n\geq1$.
\end{proof}
Let $M$ be the set of minimal forbidden Fibonacci words. The following description of $M$ can
be found in \cite{MRS02}:
\begin{equation}\label{M}
    M = \{\;M_{2n+1}\;|\;M_{2n+1}=1p_{2n+1}1, \; n\geq1\}\cup\{\;M_{2n+2}\;|\;M_{2n+2}=0p_{2n+2}0, \; n\geq1\}
\end{equation}
The first few values of the sequence $M_n$ are recorded in Table~\ref{tab4}.
\begin{table}[H]
    \centering
    \begin{tabular}{l}
        $M_3$ = 11 \\
        $M_4$ = 000 \\
        $M_5$ = 10101 \\
        $M_6$ = 00100100 \\
        $M_7$ = 1010010100101 \\
        $M_8$ = 001001010010010100100 \\
        $M_9$ = 1010010100100101001010010010100101
    \end{tabular}
    \caption{The first few minimal forbidden Fibonacci words $M_n$.}
    \label{tab4}
\end{table}
Note that $|f_n| = F_n$ implies $|p_n| = F_n - 2$, so by construction, the $n^{th}$ minimal forbidden Fibonacci word $M_n$ has $|M_n| = F_n$. Furthermore, as each $p_n$ is a palindrome, so is each $M_n$.

Guibas and Odlyzko~\cite{GO81} introduced the notion of the \emph{correlation} of two words.


\begin{dfn}\label{correlation}
    For every pair of words $(u,v)\in A^n\times A^m$, the correlation of $u$ over $v$ is the word $C_{u,v}\in A^n$ such that for all $k\in [n]$, 
    \begin{align*} 
    C_{u,v}[k] =
    \begin{cases} 
        1 & \text{if }\; \forall i\in[n],\; j\in[m],\; \text{with } i = j + k, \\ &u[i] = v[j], \\
        0 & \text{otherwise.}
    \end{cases}
\end{align*}
\end{dfn}
\begin{table}[H]
    \centering
    \begin{tabular}{c||c|c|c|c|c|c|c|c|c|c|c|c||c}
        pos. & 0 & 1 & 2 & 3 & 4 & 5 & 6 & 7 & 8 & 9 & 10 & 11 & \\
        \hline
        $u$ & 1 & 0 & 1 & 0 & 0 & 1 & 0 & 1 & - & - & - & - & $C_{u,v}$\\
        \hline 
        $v$ & 1 & 0 & 1 & 0 & 1 & - & - & - & - & - & - & - & 0\\
          & - & 1 & 0 & 1 & 0 & 1 & - & - & - & - & - & - & 0\\
          & - & - & 1 & 0 & 1 & 0 & 1 & - & - & - & - & - & 0\\
          & - & - & - & 1 & 0 & 1 & 0 & 1 & - & - & - & - & 0\\
          & - & - & - & - & 1 & 0 & 1 & 0 & 1 & - & - & - & 0\\
          & - & - & - & - & - & 1 & 0 & 1 & 0 & 1 & - & - & 1\\
          & - & - & - & - & - & - & 1 & 0 & 1 & 0 & 1 & - & 0\\
          & - & - & - & - & - & - & - & 1 & 0 & 1 & 0 & 1 & 1\\
    \end{tabular}
    \caption{The correlation of $u=10100101$ over $v=10101$, where the $(i+3)^{th}$ row is the $i^{th}$ shift of $v$, and the $i^{th}$ letter of $C_{u,v}$ is given in the $(i+3)^{th}$ entry of the last column.}
    \label{tab5}
\end{table}
We would like to interpret the correlations $C_{u,v}$ as a polynomial over the variable $z$, which we call the $\textit{correlation polynomial}$ of $u$ over $v$, and denote it as $C_{u,v}(z)$. If $|u|=n$, then
\[
C_{u,v}(z) = \sum_{k=0}^{n-1} C_{u,v}[k]z^{n-1-k}.
\]
Thus from the example given in Table~$\ref{tab5}$, we get the polynomial $C_{u,v}(z) = 1 + z^2$.

\section{Results}
The goal of this section is to characterize the correlations between any two minimal forbidden Fibonacci words.
\begin{lem}\label{lem2}
Let $u\in A^n$, $v\in A^m$, with $n\leq m$. Then $C_{u,v}[k]=1$ if and only if $u[k,k+1,\cdots,n-1] = v[0,1,\cdots,n-k-1]$; i.e. the overlapping blocks are the suffix of $u$ and the prefix of $v$ of length $n-k$.
\end{lem}
\begin{proof}
    Consider $k\in[n]$ such that $C_{u,v}[k]=1$. Then $i=k$ is the smallest $i\in[n]$ that could satisfy $i=j+k$,  as $i<k$ and $i=j+k$ implies $j<0$, which can't happen. Then for every $k\leq r\leq n-1$, we have that $i=r$ and $i=j+k$ implies $j=r-k$, so $u[r]=v[r-k]$. So it must be that the blocks $u[k,k+1,\ldots,n-1]$ and $v[0,1,\ldots,n-k-1]$ are equal, which are the suffix of $u$ and the prefix of $v$ of length $n-k$, respectively.
\end{proof}
\begin{lem}\label{lem3}
    Let $u\in A^n$, $v\in A^m$, with $n>m$, and $v$ is not a subword of $u$. Then $C_{u,v}[k]=1$ if and only if $k>n-m$ and $u[k,k+1,\ldots,n-1] = v[0,1,\ldots,n-k-1]$; i.e. the overlapping blocks are the suffix of $u$ and the prefix of $v$ of length $n-k$.
\end{lem}
\begin{proof}
    First, consider $k\leq n-m$. Then for each $j\in[m]$, we have $i=j+k\leq j+(n-m)<n$. So in order for $C_{u,v}[k]=1$, we would need $u[k+j]=v[j]$ for each $j\in[m]$, i.e. we would need $v$ to be a subword of $u$, which is a contradiction. So $C_{u,v}[k]=0$ for $k\leq n-m$.
    When $k>n-m$, the proof is the same as in Lemma \ref{lem2}.
\end{proof}
\begin{thm}\label{thm1}
    Let $n\leq m$ and let $u\in A^n, v\in A^m$.  If $u$ and $v$ are palindromes then $C_{u,v} = C_{v,u}[m-n..m-1]$.
\end{thm}
\begin{proof}
     It is sufficient to show $C_{u,v}[k] = C_{v,u}[m-n+k]$ for each $k\in[n]$.  By Lemma~\ref{lem2}, we have $C_{u,v}[k]=1$
     if and only if $u$ has a suffix $x$ of length $n-k$ that is also a prefix of $v$.  Since $u$ and $v$ are palindromes,
     this holds if and only if $\tilde{x}$ is a prefix of $u$ and a suffix of $v$.  This existence of such an
     $\tilde{x}$ is equivalent to $C_{v,u}[m-n+k]=1$.  We get $C_{u,v}[k] = C_{v,u}[m-n+k]$, as required.     
     
    
\end{proof}
\begin{lem}\label{lem4}
    Let $\phi$ be the Fibonacci morphism. Then the following hold for $n\geq1$:
    \begin{description}
        \item[a)] $\phi(1p_{2n+2}1)0$ = $0p_{2n+3}0$ 
        \item[b)] $0^{-1}\phi(0p_{2n+1}0)$ = $1p_{2n+2}1$
        \item[c)] $\phi(M_{2n+1})0$ = $M_{2n+2}$
        \item[d)] $0^{-1}\phi(M_{2n+2})$ = $M_{2n+3}$
    \end{description}
\end{lem}
\begin{proof}
    Using Lemma~\ref{pn:phi}, we have
    \[\phi(1p_{2n+2}1)0=\phi(1)\phi(p_{2n+2})\phi(1)0=0p_{2n+3}0^{-1}00=0p_{2n+3}0,\]
    and
    \[0^{-1}\phi(0p_{2n+1}0)=0^{-1}\phi(0)\phi(p_{2n+1})\phi(0)=0^{-1}01p_{2n+2}0^{-1}01=1p_{2n+2}1.\]
    Similarly, using \eqref{M} and Lemma~\ref{pn:phi}, we have
    \[\phi(M_{2n+1})0=\phi(1p_{2n+1}1)0=\phi(1)\phi(p_{2n+1})\phi(1)0=0p_{2n+2}0^{-1}00=0p_{2n+2}0=M_{2n+2},\]
    and
    \begin{multline*}
    0^{-1}\phi(M_{2n+2})=0^{-1}\phi(0p_{2n+2}0)=0^{-1}\phi(0)\phi(p_{2n+2})\phi(0)=0^{-1}01p_{2n+3}0^{-1}01\\
    =1p_{2n+3}1=M_{2n+3}.
    \end{multline*}
\end{proof}
For the following Lemma, we will be using the inverse of $\phi$, but we have to be careful to avoid any ambiguity. For example, there is no word $u$ such that $\phi(u) = 1$, so $\phi^{-1}(1)$ does not exist. To avoid such ambiguities, we will define the inverse of $\phi$, denoted $\phi^{-1}$ as such:
     \begin{equation}\label{phi:inverse}
         u = \phi(v) \implies \phi^{-1}(u) = v,\;\text{for}\; u,v,\in\{0,1\}^*
     \end{equation}
\begin{lem}\label{lem5}
    Let $\phi$ be the Fibonacci morphism. Then the following hold for $n\geq1$:
     \begin{description}
        \item[a)]  $b$ is a border of $M_{2n+1}$ $\implies$ $\phi(b)0$ is a border of $M_{2n+2}$
        \item[b)] $b$ is a border of $M_{2n+2}$ $\implies$ $0^{-1}\phi(b)$ is a border of $M_{2n+3}$
        \item[c)] $b$ is a border of $M_{2n+3}$ $\implies$ $\phi^{-1}(0b)$ is a border of $M_{2n+2}$
        \item[d)] $b$ is a border of $M_{2n+2}$ $\implies$ $\phi^{-1}(b0^{-1})$ is a border of $M_{2n+1}$
    \end{description}
\end{lem}
\begin{proof}
    Let $b$ be a border of $M_{2n+1}$ for some $n\geq1$. Then $M_{2n+1} = bu = vb$ for some words $u$ and $v$. Since $M_n$ is a border of itself for each $M_n$, Lemma \ref{lem4} $\textbf{c)}$ confirms $\textbf{a)}$ holds for $b=M_{2n+1}$. So assume $u,v\neq\epsilon$. Then we can expand $M_{2n+2}$ two different ways,
    \begin{align*}
       M_{2n+2} &= \phi(M_{2n+1})0 = \phi(bu)0 = \phi(b)\phi(u)0 \\
        M_{2n+2} &= \phi(M_{2n+1})0 = \phi(vb)0 = \phi(v)\phi(b)0  
    \end{align*}
    and by \eqref{phi}, the first letter of $\phi(u)$ will be $0$ no matter what the first letter of $u$ is. Hence, we can see that $M_{2n+2}$ both starts and ends with $\phi(b)0$, making $\phi(b)0$ a border of $M_{2n+2}$, confirming $\textbf{a)}$.
    
    Now let $b$ be a border of $M_{2n+2}$ for some $n\geq1$. Then $M_{2n+2} = bu = vb$ for some words $u$ and $v$. Again Lemma \ref{lem4} $\textbf{d)}$ confirms that $\textbf{b)}$ holds for $b=M_{2n+2}$, so assume $u,v\neq\epsilon$. Then since each even minimal forbidden Fibonacci word starts (and ends) with a 0, we can let $b=0b'$, then $\phi(b) = 01\phi(b')$ (if $b=0$, then $b'=\epsilon$, and $\phi(\epsilon) = \epsilon$) to get
    \begin{align*}
    M_{2n+3} &= 0^{-1}\phi(M_{2n+2}) = 0^{-1}\phi(bu) = 0^{-1}01\phi(b')\phi(u) = 1\phi(b')\phi(u) \\
        M_{2n+3} &= 0^{-1}\phi(M_{2n+2}) = 0^{-1}\phi(vb) = 0^{-1}\phi(v)01\phi(b')
    \end{align*}
     Thus $M_{2n+3}$ both starts and ends with $1\phi(b') = 0^{-1}\phi(b)$, making $0^{-1}\phi(b)$ a border of $M_{2n+3}$, confirming $\textbf{b)}$.
     
     Next, let $b$ be a border of $M_{2n+3}$ for some $n\geq1$. Then $M_{2n+3} = bu = vb$ for some words $u$ and $v$. First, consider the case where $b = M_{2n+3}$. Then by Lemma \ref{lem4} $\textbf{d)}$ and \eqref{phi:inverse},  
     \begin{equation}\label{lem5eq1}
         \phi^{-1}(0M_{2n+3}) = \phi^{-1}(00^{-1}\phi(M_{2n+2})) = M_{2n+2}
     \end{equation}
     So $\textbf{c)}$ holds in this case, and we can assume $u,v\neq\epsilon$. Then since each odd minimal forbidden Fibonacci word starts (and ends) with a $1$, the first letter of $b$ will be $1$. Furthermore, since the subword $11$ is forbidden, the last letter of $v$ and the first letter of $u$ must be 0. So if we let $v=v'0$, then $M_{2n+3}$ = $v'0b$, and both $\phi^{-1}(u)$, $\phi^{-1}(0v')$ exist. Thus, by \eqref{phi:inverse} and \eqref{lem5eq1},
     \begin{align*}
        M_{2n+2} &= \phi^{-1}(0M_{2n+3}) = \phi^{-1}(0bu) = \phi^{-1}(0b)\phi^{-1}(u) \\
        M_{2n+2} &= \phi^{-1}(0M_{2n+3}) = \phi^{-1}(0vb) = \phi^{-1}(0v')\phi^{-1}(0b) 
     \end{align*}
     so $M_{2n+2}$ both starts and ends with the subword $\phi^{-1}(0b)$, making $\phi^{-1}(0b)$ a border of $M_{2n+2}$, confirming $\textbf{c)}$.
     
     Lastly, let $b$ be a border of $M_{2n+2}$ for some $n\geq1$. Then $M_{2n+2} = bu = vb$. First, consider the case where $b=M_{2n+2}$. Then by Lemma \ref{lem4} $\textbf{c)}$ and \eqref{phi:inverse}, 
     \begin{equation}\label{lem5eq2}
         \phi^{-1}(M_{2n+2}0^{-1}) = \phi^{-1}(\phi(M_{2n+1})00^{-1}) = M_{2n+1}
     \end{equation}
     So $\textbf{d)}$ holds in this case, and we can assume $u,v\neq\epsilon$. Also, $b=0$ would give $\phi(b0^{-1}) = \phi(\epsilon) = \epsilon$, which is trivially a border of $M_{2n+1}$, so we may assume $|b|>1$.
     
     As every even minimal forbidden Fibonacci word starts (and ends) with $00$, we may write $b=b'00$ (if $b=00$, then $b'=\epsilon$ and $\phi^{-1}(\epsilon)=\epsilon$), and both $\phi^{-1}(b'0)$, $\phi^{-1}(0u0^{-1})$ exist. Then by \eqref{phi:inverse} and \eqref{lem5eq2},
     \begin{align*}
         M_{2n+1} &= \phi^{-1}(M_{2n+2}0^{-1}) = \phi^{-1}(b'00u0^{-1}) = \phi^{-1}(b'0)\phi^{-1}(0u0^{-1}) \\
         M_{2n+1} &= \phi^{-1}(M_{2n+2}0^{-1}) = \phi^{-1}(vb'000^{-1}) = \phi^{-1}(v)\phi^{-1}(b'0)
     \end{align*}
     so $M_{2n+1}$ both starts and ends with the subword $\phi^{-1}(b'0) = \phi^{-1}(b0^{-1})$, making $\phi^{-1}(b0^{-1})$ a border of $M_{2n+1}$, confirming $\textbf{d)}$ and finishing the proof.
\end{proof}
For $n\geq3$, we define $B_n$ to be the set of all nonempty borders of $M_n$. For example, $B_6 = \{0, 00, 00100, M_6\}$ and $B_7 = \{1, 101, 10100101, M_7\}$.
\begin{thm}\label{thm3}
    For $n\geq1$, the set of all nonempty borders of $M_{2n+2}$ and $M_{2n+3}$ are given by:
    \begin{align*}
        B_{2n+2} &= \{0\}\cup\{\phi(b)0\;|\;b\in B_{2n+1}\} \\
        B_{2n+3} &= \{0^{-1}\phi(b)\;|\;b\in B_{2n+2}\}
    \end{align*}
    with $B_3 = \{1, 11\}$.
\end{thm}
\begin{proof}
    $\textbf{Even case}$: Consider $M_{2n+2}$, for some $n\geq1$. By construction, 0 is a border of $M_{2n+2}$, so $0\in B_{2n+2}$. Furthermore, Lemma \ref{lem5} \textbf{a)} tells us that $\phi(b)0$ is a border of $M_{2n+2}$ for each $b\in B_{2n+1}$, so 
    \begin{equation}\label{thm3eq1}
        \{0\}\cup\{\phi(b)0\;|\;b\in B_{2n+1}\} \subseteq B_{2n+2}.
    \end{equation}
    Now let $b\in B_{2n+2}$, $b\neq0$. Then $b$ must start and end with a 0, so we can write $b=b'0$. Then by Lemma \ref{lem5} $\textbf{d)}$, $a=\phi^{-1}(b0^{-1})$ is a border of $M_{2n+1}$, so $a\in B_{2n+1}$. By Lemma \ref{lem5} $\textbf{a)}$ and \eqref{phi:inverse},
    \begin{align*}
        \phi(a)
        &= \phi(\phi^{-1}(b0^{-1}))0 \\
        &= \phi(\phi^{-1}(b'))0 \\
        &= b'0 \\
        &= b
    \end{align*}
    so $b\in\{\phi(a)0\;|\;a\in B_{2n+1}\}$. But $b$ was arbitrary, therefore
    \begin{equation}\label{thm3eq2}
        B_{2n+2} \subseteq \{0\}\cup\{\phi(b)0\;|\;b\in B_{2n+1}\}.
    \end{equation}
    Thus, by \eqref{thm3eq1} and \eqref{thm3eq2}, $B_{2n+2} = \{0\}\cup\{\phi(b)0\;|\;b\in B_{2n+1}\}$.\\ \\
    $\textbf{Odd case}$: Consider $M_{2n+3}$, for some $n\geq1$. Then by Lemma \ref{lem5} $\textbf{b)}$, $0^{-1}\phi(b)$ is a border of $M_{2n+3}$ for each $b\in B_{2n+2}$, so 
    \begin{equation}\label{thm3eq3}
        \{0^{-1}\phi(b)\;|\;b\in B_{2n+2}\} \subseteq B_{2n+3} 
    \end{equation}
    Now let $b\in B_{2n+3}$. Then $b$ must start and end with a 1, so we can write $b=1b'$. Then by Lemma \ref{lem5} $\textbf{c)}$, $a=\phi^{-1}(0b)$ is a border of $M_{2n+2}$, so $a\in B_{2n+2}$. By Lemma \ref{lem5} $\textbf{b)}$ and \eqref{phi:inverse},
    \begin{align*}
        \phi(a)
        &= 0^{-1}\phi(\phi^{-1}(0b)) \\ 
        &= 0^{-1}\phi(\phi^{-1}(01b')) \\ 
        &= 0^{-1}\phi(0\phi^{-1}(b')) \\
        &= 0^{-1}01\phi(\phi^{-1}(b')) \\ 
        &= 1b' \\ 
        &= b
    \end{align*}
    so $b\in\{0^{-1}\phi(a)\;|\;a\in B_{2n+2}\}$. But $b$ was arbitrary, therefore
    \begin{equation}\label{thm3eq4}
        B_{2n+3} \subseteq \{0^{-1}\phi(a)\;|\;a\in B_{2n+2}\}
    \end{equation}
    Thus, by \eqref{thm3eq3} and \eqref{thm3eq4}, $B_{2n+3} = \{0^{-1}\phi(a)\;|\;a\in B_{2n+2}\}$.
\end{proof}
\begin{thm}\label{thm4}
    For $n\geq1$, the set of all nonempty borders of $M_{2n+2}$ and $M_{2n+3}$ are given by:
    \begin{align*}
        B_{2n+2} &= \{0\}\cup\{M_{2n+2}\}\cup\{0p_{2(n-i)+1}0\;|\;i\in[n]\} \\
        B_{2n+3} &= \{1\}\cup\{M_{2n+3}\}\cup\{1p_{2(n-i)+2}1\;|\;i\in[n]\}
    \end{align*}
    with $B_3 = \{1,11\}$.
\end{thm}
\begin{proof}
    We will prove this by induction on $n$. \\ \\
    $\textbf{Base case, n=1}$:
        \begin{itemize}
            \item $M_4=000$ has borders $0$, $00$, and $000$, so $B_4 = \{0,00,000\}$. But $p_3=\epsilon$, so $0p_30 = 00$, thus $\{0\}\cup\{M_4\}\cup\{0p_{2(1-i)+1}0\;|\;i\in[1]\} = \{0,00,000\}=B_4$.
            \item $M_5 = 10101$ has borders $1$, $101$, $10101$, so $B_5 = \{1,101,10101\}$. But $p_4=0$, so $1p_41=101$, thus $\{1\}\cup\{M_5\}\cup\{0p_{2(1-i)+2}0\;|\;i\in[1]\} = \{1,101,10101\} = B_5$, which confirms the base case.
        \end{itemize}
    $\textbf{Inductive step}$: Let $k>1$ and assume the statement holds up to and including $n=k-1$. \\ \\
    We start with the even word, $M_{2k+2}$. As all even minimal forbidden Fibonacci words start and end with 0, the word 0 is a border of $B_{2k+2}$. By the inductive hypothesis, all borders of $M_{2(k-1)+3} = M_{2k+1}$ are given by the set $B_{2k+1} = \{1\}\cup\{M_{2k+1}\}\cup\{1p_{2(k-i)}1\;|\;i\in[k-1]\}$. By Theorem \ref{thm3}, we also have that $B_{2k+2} = \{0\}\cup\{\phi(b)0\;|\;b\in B_{2k+1}\}$. So using Lemma \ref{lem4},
    \begin{align*}
        \phi(1)0 &= 00 = 0p_30 \\
        \phi(M_{2k+1})0 &= M_{2k+2} \\
        \phi(1p_{2(k-i)}1)0 &= 0p_{2(k-i)+1}0,\;\; \forall i\in[k-1].
    \end{align*}
    Note that $p_3 = p_{2(k-(k-1))+1}$, and thus
    \begin{align*}
        B_{2k+2} = \{0\}\cup\{\phi(b)0\;|\;b\in B_{2k+1}\} = \{0\}\cup\{M_{2k+2}\}\cup\{0p_{2(k-i)+1}0\;|\;i\in[k]\}.
    \end{align*}
    Next, by Theorem \ref{thm3}, we know $B_{2k+3} = \{0^{-1}\phi(b)\;|\;b\in B_{2k+2}\}$. So using Lemma \ref{lem4} again,
    \begin{align*}
        0^{-1}\phi(0) &= 0^{-1}01 = 1 \\
        0^{-1}\phi(M_{2k+2}) &= M_{2k+3} \\
        0^{-1}\phi(0p_{2(k-i)+1}0) &= 1p_{2(k-i)+2}1,\;\; \forall i\in[k]
    \end{align*}
    and thus 
    \begin{align*}
        B_{2k+3} =  \{0^{-1}\phi(b)\;|\;b\in B_{2k+2}\} = \{1\}\cup\{M_{2k+3}\}\cup\{1p_{2(k-i)+2}1\;|\;i\in[k]\}.
    \end{align*}
\end{proof}
\begin{thm}\label{thm5}
    Let $3\leq n\leq m$. Then
    \begin{align*}
        C_{M_n,M_m}(z) = C_{M_m,M_n}(z) = \sum_{b\in B_n\cap B_m}z^{|b|-1}.
    \end{align*}
\end{thm}
\begin{proof}
    \textbf{Case 1}, $n=m$: By Lemma \ref{lem2}, $C_{M_n,M_n}[k] = 1$ if and only if the suffix of $M_n$ is identical to the prefix of $M_n$ of length $F_n-k$. In other words,  $C_{M_n,M_n}[k] = 1$ exactly when $M_n$ has a border of length $F_n-k$.
    
    Recall that $C_{M_n,M_n}(z)$ has a term $z^{k-1}$ only if the $(F_n-1)-(k-1)^{th}$ letter of $C_{M_n,M_n}$ is 1, in which case $M_n$ has a border of length 
    \begin{align*}
        F_n - ((F_n-1)-(k-1)) = k.
    \end{align*}
    Thus for each $b\in B_n$, $C_{M_n,M_n}(z)$ has a term $z^{|b|-1}$, so
    \begin{align*}
        C_{M_n,M_n}(z) = \sum_{b\in B_n} z^{|b|-1}.
    \end{align*}
    
    \noindent\textbf{Case 2}, $n\leq m$: By Theorem \ref{thm1}, we know $C_{M_n,M_m} = C_{M_m,M_n}[F_m-F_n, F_m-F_n+1, \ldots, F_m-1]$. But since these are both \textit{minimal} forbidden Fibonacci words, each proper factor of $M_m$ will be a factor of \textbf{f}. But $M_n$ is not a factor of \textbf{f}, and therefore cannot be a factor of $M_m$. Thus by Lemma \ref{lem3}, $C_{M_m,M_n}[k]=0$ for $0\leq k\leq F_m-F_n$.
    
    Furthermore, since $|M_n|<|M_m|$, $C_{M_m,M_n}[k]=1$ if and only if the suffix of $M_m$ is identical to the prefix of $M_n$ of length $F_m-k$, with $k>F_m-F_n$. So if $k=F_m-F_n+k'$, with $k'\in[F_n]$, then this length is $F_m-(F_m-F_n+k')=F_n-k'$. Moreover, by Theorem \ref{thm1}, for each $k'\in[F_n]$, $C_{M_n,M_m}[k'] = C_{M_m,M_n}[F_m-F_n+k'] = C_{M_m,M_n}[k]$.
    
    By Lemma \ref{lem2}, $C_{M_n,M_m}[k']=1$ if and only if the suffix of $M_n$ is identical to the prefix of $M_m$ of length $F_n-k'$. So $C_{M_n,M_m}[k']=C_{M_m,M_n}[k]=1$ exactly when the prefix and the suffix of length $F_n-k'$ of both $M_n$ and $M_m$ are equal, i.e. if they share a border of length $F_n-k'$.
    
    Additionally, $C_{M_n,M_m}(z)$ has a term $z^{k-1}$ only if $C_{M_n,M_m}[(F_n-1)-(k-1)]=1$, which happens when $M_n$ and $M_m$ share a border of length 
    \begin{align*}
         F_n - ((F_n-1)-(k-1)) = k.
    \end{align*}
    But by Theorem \ref{thm1},
    \begin{align*}
        C_{M_n,M_m}[F_n-1-(k-1)] 
        &= C_{M_m,M_n}[F_m-F_n+(F_n-1-(k-1))] \\
        &= C_{M_m,M_n}[F_m-1-(k-1)] \\
        &= 1.
    \end{align*}
    So $C_{M_m,M_n}(z)$ also has a term $z^{k-1}$, thus $C_{M_n,M_m}(z) = C_{M_m,M_n}(z)$.
    
    Finally, for each $b\in B_n\cap B_m$, $C_{M_n,M_m}(z)$ has a term $z^{|b|-1}$, so
    \begin{align*}
        C_{M_n,M_m}(z) = C_{M_m,M_n}(z) = \sum_{b\in B_n\cap B_m}z^{|b|-1}.
    \end{align*}
\end{proof}

Note that by Theorem \ref{thm4}, for $1\leq n<m$, we have
\begin{align*}
    B_{2n+1}\cap B_{2m+1} &= B_{2n+1}\setminus\{M_{2n+1}\} \\
    B_{2n+2}\cap B_{2m+2} &= B_{2n+2}\setminus\{M_{2n+2}\},
\end{align*}
and for $3\leq n<m$, if $n\not\equiv m \pmod{2}$, then 
\begin{align*}
    B_n\cap B_m = \emptyset,
\end{align*}
which leads us to our final result.

\begin{thm}\label{thm6}
    Let $1\leq n<m$. Then the correlation polynomials between minimal forbidden Fibonacci words are given as follows:
    \begin{align*}
        C_{M_{2n+1},M_{2n+1}}(z) &= z^{F_{2n+1}-1} + \sum_{i=0}^{n-1}z^{F_{2(n-i)}-1} \\
        C_{M_{2n+2},M_{2n+2}}(z) &= z^{F_{2n+2}-1} + \sum_{i=0}^{n}z^{F_{2(n-i)+1}-1} \\
        C_{M_{2n+1},M_{2m+1}}(z) &= C_{M_{2m+1},M_{2n+1}}(z) = \sum_{i=0}^{n-1}z^{F_{2(n-i)}-1} \\
        C_{M_{2n+2},M_{2m+2}}(z) &= C_{M_{2m+2},M_{2n+2}}(z) = \sum_{i=0}^{n}z^{F_{2(n-i)+1}-1} \\
        C_{M_{2n+1},M_{2m+2}}(z) &= C_{M_{2m+2},M_{2n+1}}(z) = 0 \\
        C_{M_{2m+1},M_{2n+2}}(z) &= C_{M_{2n+2},M_{2m+1}}(z) = 0
    \end{align*}
\end{thm}
\begin{proof}
    Follows immediately from Theorem \ref{thm4} and \ref{thm5}.
\end{proof}

\begin{ex}
For example, let $M$ be the $6 \times 6$ matrix with $(i,j)$-entry equal to $C_{M_i,M_j}$
for $1 \leq i,j \leq 6$.  Then
$$
M=
\footnotesize
\left(\begin{array}{rrrrrr}
z + 1 & 0 & 1 & 0 & 1 & 0 \\
0 & z^{2} + z + 1 & 0 & z + 1 & 0 & z + 1 \\
1 & 0 & z^{4} + z^{2} + 1 & 0 & z^{2} + 1 & 0 \\
0 & z + 1 & 0 & z^{7} + z^{4} + z + 1 & 0 & z^{4} + z + 1 \\
1 & 0 & z^{2} + 1 & 0 & z^{12} + z^{7} + z^{2} + 1 & 0 \\
0 & z + 1 & 0 & z^{4} + z + 1 & 0 & z^{20} + z^{12} + z^{4} + z + 1
\end{array}\right).
$$
\end{ex}

\end{document}